\documentclass[12pt]{amsart}

\usepackage{amsmath,amsthm,amscd,amsfonts,amssymb,enumerate}
\usepackage{graphicx}
\usepackage{color}
\usepackage[colorlinks]{hyperref}
\usepackage{amsfonts,amssymb,amscd,amsmath,enumerate,url,verbatim}
 \usepackage[dvips]{epsfig}
 \usepackage[none]{hyphenat}
\usepackage{amsmath,amssymb,amsfonts,enumerate,amsthm}
 \usepackage{amsgen, amstext,amsbsy,amsopn, amsthm, amsfonts,amssymb,amscd,amsmat
 h,euscript,enumerate,url,verbatim,calc,xypic}
 \usepackage{latexsym}
 \usepackage{graphics}
 \usepackage{color}
\headsep=1cm \numberwithin{equation}{section}
\numberwithin{equation}{section}

\numberwithin{equation}{section}
\input xy
\xyoption{all}

\newtheorem{thm}{Theorem}[section]
\newtheorem{cor}[thm]{Corollary}
\newtheorem{con}[thm]{Convetion}
\newtheorem{lem}[thm]{Lemma}
\newtheorem{prop}[thm]{Proposition}
\newtheorem{defn}[thm]{Definition}

\newtheorem{exam}[thm]{Example}
\newtheorem{rem}[thm]{Remark}

\newtheorem{note}[thm]{Notation}

\newcommand{\coker}{\operatorname{Coker}\,}
\newcommand{\Hom}{\operatorname{Hom}\,}
\newcommand{\Ext}{\operatorname{Ext}\,}

\newcommand{\Spec}{\operatorname{Spec}\,}

\newcommand{\Ker}{\operatorname{Ker}\,}
\newcommand{\Ass}{\operatorname{Ass}\,}

\newcommand{\Supp}{\operatorname{Supp}\,}

\renewcommand{\dim}{\operatorname{dim}\,}

\renewcommand{\deg}{\operatorname{deg}\,}

\newcommand{\Tr}{\operatorname{Tr}\,}

\newcommand{\reg}{\operatorname{reg}\,}
\newcommand{\proj}{\operatorname{Proj}\,}

\newcommand{\fa}{\mathfrak{a}}

\newcommand{\fp}{\mathfrak{p}}

\begin{document}
\bibliographystyle{amsplain}

\title[Linked sheaves of modules]
 {Linked sheaves of modules}

\address{Faculty of mathematical sciences and computer,
amirkabir university and Iran National Science Foundation, Tehran, iran.} 

\email{frahmati@aut.ac.ir}

\email{sayyarikh@gmail.com}
\bibliographystyle{amsplain}

     \author[F. Rahmati]{Farhad Rahmati}
     \author[kh. sayyari]{khadijeh sayyari}





\begin{abstract} 
We introduce a notion of linkage for sheaves of modules on connected schemes, extending classical linkage of modules. Linkage is defined for stable sheaves admitting finite free resolutions via the transpose and syzygy functors. We show that linkedness is a local property and that, on affine schemes, a coherent sheaf is linked if and only if its module of global sections is linked.

We further show that linkage is preserved under restriction and, under suitable rank conditions, under gluing over connected schemes. We also obtain criteria for the existence of linked subsheaves, when the structure sheaf is not a domain. In the projective setting, we compare invariants of linked sheaves, including graded cohomology modules, Castelnuovo–Mumford regularity, and Hilbert polynomials.
\end{abstract}

\maketitle

\bibliographystyle{amsplain}
 \section{Introduction}
Linkage theory has its origins in classical algebraic geometry, beginning with the work of Halphen (1870) and M. Noether \cite{No} (1882), on the classification of algebraic curves. A decisive modern development was given by Peskine and Szpiro, \cite{PS} (1974), who reformulated linkage in algebraic terms by introducing linkage of ideals in Cohen–Macaulay local rings via regular sequences. Since then, linkage has become an important tool in commutative algebra and algebraic geometry, closely connected to homological methods, syzygies, and duality.

The theory was later extended from ideals to finitely generated modules, notably through the work of Auslander and Bridger \cite{A} and further developments by Yoshino \cite{Y}, Martsinkovsky and Strooker \cite{MS},  and others. These extensions revealed that linkage reflects deep structural relations between modules. The concept of linkage of modules is defined as follows;

Let $M$ be a finitely generated $R$-module and 
$\overset{l_2}{\oplus}R \overset{\phi}{\rightarrow} \overset{l_1}{\oplus}R\overset{\varphi}{\rightarrow} M \rightarrow 0$
be a minimal free resolution of $M$. 
By applying $(-)^+ := \Hom_R(- ,R)$, we obtain 
the exact sequences
$0 \rightarrow M^+ \overset{\varphi^+}{\rightarrow} 
(\overset{l_1}{\oplus}R)^+ \overset{\phi^+}{\rightarrow} (\overset{l_2}{\oplus}R)^+ \rightarrow \Tr M \rightarrow 0$
and
$0 \rightarrow M^+ \overset{\varphi^+}{\rightarrow} (\overset{l_1}{\oplus}R)^+\rightarrow \lambda M \rightarrow 0,$
where $\Tr M:= \coker(\phi^*)$,
$\lambda M := \coker(\varphi^+) = \Omega(\Tr M)$ and $\Omega$ denotes the first syzygy. 
Now, following \cite{MS}, two finitely generated $R$-modules $M$ and $N$ are linked, denoted by
$M \sim N$, if $M \cong \lambda N$ and $N \cong \lambda M$. 

From a geometric viewpoint, this naturally leads to the problem of formulating and studying linkage in the category of sheaves of modules on schemes.
The aim of this paper is to develop a systematic theory of linkage of sheaves of modules, viewed as a sheaf-theoretic generalization of linkage of modules.
Throughout the paper, we work over connected schemes and consider coherent sheaves (or, more generally, sheaves admitting finite free resolutions). Using the transpose and syzygy constructions, we introduce a notion of linkage for stable sheaves that extends the classical definition for modules (Definition \ref{B.2}).

A first group of results concerns the functors involved in the definition of linkage. In Sections 2 and 3, we study the transpose and syzygy functors for sheaves of modules and show that they behave well with respect to restriction to open subsets. In particular, we prove that these constructions commute with localization (Theorem \ref{B5}) and that a coherent sheaf is locally free if and only if its transpose vanishes (Proposition \ref{B4}). These results ensure that the definition of linkage is well behaved in the sheaf-theoretic setting.

One of the main results of the paper is that linkedness of sheaves is a local property. More precisely, we show that a sheaf of modules is linked if and only if its restriction to a special open subset is linked
(Theorem \ref{C2.3}). As a consequence, when the scheme is affine, linkage of sheaves is equivalent to linkage of their modules of global sections (Theorem \ref{C1} and Corollary \ref{D}). This result provides a precise bridge between classical linkage of modules and linkage of sheaves.

Another central theme of the paper is the behavior of linkage under gluing constructions. We prove that, under suitable rank conditions, a sheaf obtained by gluing a family of linked sheaves over connected schemes is itself linked (Corollary \ref{C}). These results show that linkage is compatible with standard geometric constructions and allows the construction of global linked sheaves from local data.

We also investigate the existence of linked subsheaves. In particular, for open subsets whose structure sheaf is not an integral domain, we show that a sheaf of modules admits a nonzero linked subsheaf if and only if a natural condition involving associated primes is satisfied (Proposition \ref{L1}). 

In the final section, we study linkage of sheaves on projective schemes. We compare invariants of linked sheaves such as graded cohomology modules, Castelnuovo–Mumford regularity, Hilbert polynomials, degrees, and indices of regularity. In particular, we relate the regularity of a sheaf to that of its linked sheaf via the transpose and syzygy constructions (Theorem \ref{D3} and Proposition \ref{D1}) and extend known results on degrees and Hilbert polynomials of linked modules to the sheaf-theoretic context (Corollary \ref{D2}).

The paper is organized as follows. In Section 2, we introduce the transpose and syzygy functors for sheaves of modules. Section 3 is devoted to their main properties and their relation to locally free sheaves. In Section 4, we define linked sheaves and establish the locality, gluing, and existence results. Finally, Section 5 is devoted to linkage on projective schemes and the comparison of invariants of linked sheaves.
  
Thorough out the text, we assume that $X$ is a connected scheme and $\mathfrak{F}$ is a sheaf of $\mathcal{O}_X$-modules such that $\mathfrak{F}$ is a coherent sheaf or $\mathcal{O}_X$-finitely presented.
Also, $R$ is a commutative Noetherian ring with $1\neq 0,$ all $R$-modules are finitely generated and $k$ is a field.
\section{Some functors over the category of sheaves of modules}
In this section, inspired by the module case, 
we introduce the operations $\overline{\Tr}$ and $\overline{\lambda}$
using free resolutions of sheaves and show that they define functors on the category of coherent sheaves of modules on an affine scheme. To this end,
we recall some definitions.
\begin{defn}\label{B1}
Let 
$\overset{t_2}{\oplus}\mathcal{O}_X 
\overset{\phi}{\rightarrow} 
\overset{t_1}{\oplus}\mathcal{O}_X 
\overset{\varphi}{\rightarrow}
\mathfrak{F} \rightarrow 0$
be a free resolution of $\mathfrak{F}$. We say that $\mathfrak{F}$ is $\mathcal{O}_X$-finitely 
presented. Also, we say $\mathfrak{F}$ has $(t_1 , t_2)$-rank if the pair \((t_1,t_2)\) is lexicographically minimal among all pairs \((l_1,l_2)\) arising from free resolutions of \(\mathfrak F\), namely, for any other free resolution $\overset{l_2}{\oplus}\mathcal{O}_X 
\rightarrow\overset{l_1}{\oplus}\mathcal{O}_X 
\rightarrow\mathfrak{F} \rightarrow 0$, we have either $l_1> t_1$, or $l_1= t_1$ and $l_2\geq t_2$.

Suppose that $\mathfrak{F}$ has $(t_1 , t_2)$-rank. By applying $(-)^* := \mathcal{H}om_{\mathcal{O}_X}(- ,\mathcal{O}_X)$, we obtain 
the exact sequence
\begin{equation}\label{e2}
0 \longrightarrow 
\mathfrak{F}^* 
\overset{\varphi^*}{\longrightarrow} 
(\overset{t_1}{\oplus}\mathcal{O}_X)^* 
\overset{\phi^*}{\longrightarrow} 
(\overset{t_2}{\oplus}\mathcal{O}_X)^* 
\longrightarrow 
\overline{\Tr}\,\mathfrak{F} 
\longrightarrow 0,
\end{equation}
where $\overline{\Tr}\,\mathfrak{F}$, the transpose of $\mathfrak{F}$, denotes 
$\coker(\phi^*)$. Similarly, we define
\[
\overline{\lambda}\,\mathfrak{F} := \coker(\varphi^*) = \Omega(\overline{\Tr}\,\mathfrak{F}),
\]
where $\Omega$ denotes the first syzygy. Hence, we obtain the further exact 
sequences
\begin{equation}
0 \longrightarrow 
\mathfrak{F}^* 
\overset{\varphi^*}{\longrightarrow} 
(\overset{t_1}{\oplus}\mathcal{O}_X)^*
\longrightarrow 
\overline{\lambda}\,\mathfrak{F} 
\longrightarrow 0
\end{equation}
and
\begin{equation}\label{e3}
0 \longrightarrow 
\overline{\lambda}\,\mathfrak{F} 
\overset{\phi^*}{\longrightarrow} 
(\overset{t_2}{\oplus}\mathcal{O}_X)^* 
\longrightarrow 
\overline{\Tr}\,\mathfrak{F} 
\longrightarrow 0.
\end{equation}
\end{defn}
By this definition, the following remark is notable.
\begin{rem}\label{B21}
\begin{itemize}
\item[(1)]
 Via \cite[Defenition 2,5.16.1, Page 121]{H}, $\mathfrak{F}$ is $\mathcal{O}_X$-finitely presented if and only if $\mathfrak{F}$ and $\Omega\mathfrak{F}$ are generated by a finite number of global sections.
 \item [(2)]
By \cite[Exercise 2,1.2]{H}, 
$\mathfrak{F}^*$, $\overline{\Tr}\,\mathfrak{F}$, and 
$\overline{\lambda}\,\mathfrak{F}$ are sheaves of 
$\mathcal{O}_X$-modules. Moreover, in view of \cite[17,10.1]{S} and \cite[Proposition 2,5.7]{H}, $\mathfrak{F}$ is quasi-coherent. Consequently, 
 the sheaves 
$\overline{\Tr}\,\mathfrak{F}$ and $\overline{\lambda}\,\mathfrak{F}$ 
are also quasi-coherent. In particular, whenever $X$ is locally Noetherian, $\mathfrak F$ is coherent by \cite[17,12.5]{S}, and hence so are $\overline{\Tr}\mathfrak F$ and $\overline{\lambda}\mathfrak F.$

\item[(3)]
If $\mathfrak{F}$ is free, then $\varphi$ is an isomorphism; hence both 
$\overline{\Tr}\,\mathfrak{F}$ and $\overline{\lambda}\,\mathfrak{F}$ vanish. 
\end{itemize}
\end{rem}
For a finitely generated module $L$, we denote by $\mu(L)$ the minimum cardinality of a finite generating set of $L.$ Similarly, for a sheaf $\mathfrak F$ generated by finitely many global sections, we denote by $\mu(\mathfrak F)$ the minimum cardinality of a finite set of global generators of $\mathfrak F.$
\begin{exam}
Let $X=\Spec(R)$ be an affine Noetherian scheme and
$\fa$ be an ideal of $R$.
Then
$\widetilde{(\frac{R}{\fa})}$ is $\mathcal{O}_X$-finitely presented and it has $(1 , \mu (\fa))$-rank. Also, $\overline{\lambda}\widetilde{R/\fa}\cong \widetilde{R/(0:_R\fa)}.$
\end{exam}
An $R$-module $M$ is said to be stable if it has no free summand.  
Two $R$-modules $M$ and $N$ are said to be stably isomorphic, denoted
$\underline{M} \cong \underline{N}$,
if there exist free $R$-modules of finite rank $H$ and $W$ such that
$H \oplus M \cong N \oplus W.$
Inspired by this notion, we define the corresponding concept for sheaves.  
\begin{defn}
A sheaf 
$\mathfrak{F}$ is called stable if there does not exist a split exact sequence
$0 \longrightarrow\overset{t}{\oplus}\mathcal{O}_X \longrightarrow \mathfrak{F}.$
Equivalently, $\mathfrak{F}$ has no free direct summand.
Two sheaves $\mathfrak{F}$ and $\mathfrak{G}$ are stability isomorphic, written
$\underline{\mathfrak{F}} \cong \underline{\mathfrak{G}},$
if there exist free sheaves of finite rank $\mathfrak{h}$ and $\mathfrak{w}$ such that
$\mathfrak{h} \oplus \mathfrak{G} \cong \mathfrak{F} \oplus \mathfrak{w}.$ Similarly, $\mathfrak{F}$ and $\mathfrak{G}$ are said to be locally stability isomorphic
if $\underline{\mathfrak{F}_{\fp}} \cong \underline{\mathfrak{G}_{\fp}}$ for every $\fp\in X$.
\end{defn}
The following lemma investigates whether the stability of two modules transfers to the stability of certain special sheaves.
\begin{lem}\label{B}
 Let $X=\Spec(R)$ be an affine scheme and $M$ and $N$ be $R$-modules. 
 Then the following statements are equivalent. 
 \begin{enumerate}
 \item [(i)]  $M$ and $N$ are stability isomorphic.
  \item [(ii)]  $\overset {\sim} M$ and $\overset {\sim} N$ are stability isomorphic.
  \item [(iii)] $\overline{\Tr}\overset {\sim} M$ and $\overline{\Tr}\overset {\sim} N$ are stability isomorphic.
\end{enumerate}
In particular, $M$ is stable if and only if $\widetilde{M}$ is stable.
\end{lem}
It is possible for two sheaves to be stability isomorphic but not isomorphic.
\begin{exam}
\begin{enumerate}
\item[(i)]
Let $X = \Spec(R)$ be an affine scheme and let $\fa$ be an ideal of $R$.  
Then $(R/\fa)^{\sim}$ is stable. Set
$\mathfrak{F} := (R/\fa)^{\sim}$ and 
$\mathfrak{G} := (R/\fa)^{\sim} \oplus \mathcal{O}_X .$
It is clear that $\mathfrak{F}$ and $\mathfrak{G}$ are stably isomorphic, but they are not isomorphic.

\item[(ii)]
Let $G=<a_1,..., a_n>$ be an $\mathbb{Z}$-group such that $n>1$ and the order of $a_1$, $O(a_1),$ be infinity. Hence $<a_1>\cong \mathbb{Z}$ and so $G$ is not stable.
\end{enumerate}
\end{exam}

There are several relations between the transpose of a module and the transpose of its associated sheaf, as demonstrated below.

\begin{thm}\label{B3}
Let $X = \Spec(R)$ be an affine scheme and $M$ be an $R$-module. Then the following statements hold:
\begin{enumerate}
    \item $\underline{\overline{\lambda}\,\widetilde{M}} \cong \underline{\widetilde{\lambda M}}$.
    \item $\underline{\overline{\Tr}\,\widetilde{M}} \cong \underline{\widetilde{\Tr M}}$.
    \item $\underline{\Gamma(X,\,\overline{\lambda}\,\widetilde{M})}
           \cong \underline{\Gamma(X,\,\widetilde{\lambda M})}
           = \underline{\lambda M}$.
    \item $\underline{\Gamma(X,\,\overline{\Tr}\,\widetilde{M})}
           \cong \underline{\Gamma(X,\,\widetilde{\Tr M})}
           = \underline{\Tr M}$.
    \item $\underline{\overline{\Tr}(\overline{\Tr}\,\widetilde{M})}
           \cong \underline{\widetilde{M}}$.
\end{enumerate}
In particular, whenever $M$ is stable, any of the above stability isomorphisms is a natural isomorphism.
\end{thm}

\begin{proof}
First, let us show that stability is an important property. For this, one can assume that $M$ is stable and compare $\lambda M$ and $\lambda(M\oplus R)$. Let 
\begin{equation}\label{e}
F_2 \xrightarrow{f} F_1 \xrightarrow{g} M \to 0.
\end{equation}
be a minimal free resolution of $M$.Thus, we obtain the exact sequence 
\begin{equation}\label{e'}
F_2\oplus R \xrightarrow{f\oplus id_R} F_1\oplus R \xrightarrow{g\oplus id_R} M\oplus R \to 0,
\end{equation}
which is a minimal free resolution of $M\oplus R$.
 Applying $(-)^+:=\Hom_R (-,R)$ to \eqref{e} and \eqref{e'} gives the following exact sequences
\begin{equation}\label{e1}
0 \to M^+ \xrightarrow{g^+} (F_1)^+ \to \lambda M \to 0,
\end{equation}
and
$$0 \to M^+\oplus R \xrightarrow{g^+\oplus id_R} (F_1)^+\oplus R \to \lambda (M\oplus R) \to 0.$$ Now, by the definition and the Five Lemma, $\lambda (M\oplus R) \cong \lambda(M)\oplus R$ and hence $\underline{\lambda (M\oplus R)} \cong \underline{\lambda(M)}.$ Therefore, via Lemma~\ref{B}, we may assume that $M$ is stable.
\begin{itemize}
    \item[(1)]
    Applying the tilde functor to \eqref{e1}, and using \cite[Propositions.\ 2.5.5 and 2.5.2]{H}, we obtain
    \[
    0 \to 
    \mathcal{H}om_{\mathcal{O}_X}(\widetilde{M},\mathcal{O}_X)
    \to 
    \mathcal{H}om_{\mathcal{O}_X}(\widetilde{F_1},\mathcal{O}_X)
    \to 
    \widetilde{\lambda M}
    \to 0.
    \]
    The claim follows from 
    the definition \ref{B1} and the Five Lemma.   
    \item[(2)]
Consider the following exact sequence
\begin{equation}\label{e.3}
0 \to \lambda M \xrightarrow{f^+} (F_2)^+ \to \Tr M \to 0.
\end{equation}
    Applying the same argument to \eqref{e.3}, together with the isomorphism in (1) and the Five Lemma, proves the result.
    \item[(3) and (4)]
    These follow from \cite[Proposition\ 2.5.1]{H} together with parts (1) and (2).
    \item[(5)]
    By \cite[2.6]{A}, the identity $\Tr \Tr(-) \cong \mathrm{id}$ holds in the category of finitely generated modules.  
    Combined with the previous items, this yields
   $\overline{\Tr}(\overline{\Tr}\,\widetilde{M}) \cong \widetilde{M}.$
\end{itemize}
The last claim follows from the proof together with \ref{B}.
\end{proof}
In view of the definition of the ranks of $\mathfrak{F}$, the objects $\overline{\Tr}\,\mathfrak{F}$  and $\overline{\lambda}\,\mathfrak{F}$ are uniquely defined. However, it is natural to ask whether a different choice of free 
resolution of $\mathfrak{F}$ effects these operators, and whether 
they are independent of the choice of free resolution.
As a consequence of Theorem \ref{B3}, the following lemma shows that, when $X$ is an affine scheme, this is indeed true up to ``stabilization.” We will return to this question in Section~3.
\begin{lem}\label{RR}
Let $X$ be an affine scheme. Then the functor $\overline{\Tr}$ descends to the stabilization of the category of coherent sheaves, sending a coherent sheaf $\mathfrak{F}$ to $\overline{\Tr}(\mathfrak{F})$.  
In particular, the functor $\overline{\lambda} := \Omega\,\overline{\Tr}$ is well defined on the stabilization as well. 
\end{lem}
\begin{proof}
It immediately follows from \cite[2.5]{A}, \cite[2.5.6]{H} and \ref{B3}.
\end{proof}
What is the relationship between the functors $\overline{\Tr}$ and $\overline{\lambda}$ and other functors? The next proposition answers this question.
\begin{prop}\label{B2}
Let $X = \Spec(R)$ be an affine scheme and $\mathfrak{F}$ be a coherent sheaf. There exists a natural exact sequence of functors on the category of coherent sheaves:
$$ 0 \rightarrow 
\mathcal{E}xt^1_{\mathcal{O}_X}(\overline{\Tr} \mathfrak{F}, -)
   \rightarrow \mathfrak{F} \otimes_{\mathcal{O}_X} - 
   \rightarrow \mathcal{H}om_{\mathcal{O}_X}(\mathfrak{F}^*, -) 
   \rightarrow \mathcal{E}xt^2_{\mathcal{O}_X}(\overline{\Tr}\mathfrak{F}, -) 
   \rightarrow 0.$$
\end{prop}
\begin{proof}
Consider two coherent sheaves $\mathfrak{h}$ and $\mathfrak{G}$ with a morphism 
$f : \mathfrak{h} \to \mathfrak{G}$. One can assume that $\mathfrak{F} = \widetilde{M}$ (respectively $\mathfrak{G} = \widetilde{N}$ and $\mathfrak{h} = \widetilde{H}$) where $M$ is a finitely generated $R$-module (respectively $N$ and $H$).
By \cite[2.6]{A}, one has the following commutative diagram:
\[
\begin{CD}
0 \rightarrow \Ext^1_R(\Tr M, H) 
   @>>> M \otimes_R H 
   @>>> \Hom_R(M^+, H) 
   @>>> \Ext^2_R(\Tr M, H) \rightarrow 0 \\
@VVV @VVV @VVV @VVV \\
0 \rightarrow \Ext^1_R(\Tr M, N) 
   @>>> M \otimes_R N 
   @>>> \Hom_R(M^+, N) 
   @>>> \Ext^2_R(\Tr M, N) \rightarrow 0,
\end{CD}
\]
where $(-)^+ := \Hom_R(-,R)$.
This and \ref{B3} imply the corresponding commutative diagram of sheaves:
 $$\begin{CD}
         &&&&&&&&\\
         \ \ &&&& 0 \rightarrow\mathcal{E}xt^1_{\mathcal{O}_X}(\overline{\Tr}\,\mathfrak{F}, \mathfrak{h}) @>>>\mathfrak{F}\otimes_{\mathcal{O}_X} \mathfrak{h} @>>>\mathcal{H}om_{\mathcal{O}_X}(\mathfrak{F}^* ,\mathfrak{h} )@>>>\mathcal{E}xt^2_{\mathcal{O}_X}(\overline{\Tr}\,\mathfrak{F} , \mathfrak{h}) \rightarrow 0&  \\
          &&&&  @VV{\xi}V   @VV{id_{\mathfrak{F}}\otimes f}V   @VV{\mathcal{H}om(id_{\mathfrak{F}},f)}V  @VVV\\
         \ \  &&&& 0 \rightarrow \mathcal{E}xt^1_{\mathcal{O}_X}(\overline{\Tr}\,\mathfrak{F},\mathfrak{G}) @>>>\mathfrak{F}\otimes_{\mathcal{O}_X}\mathfrak{G} @>>>\mathcal{H}om_{\mathcal{O}_X}(\mathfrak{F}^*,\mathfrak{G})@>>>\mathcal{E}xt^2_{\mathcal{O}_X}(\overline{\Tr}\,\mathfrak{F} , \mathfrak{G})  \rightarrow 0.&\\
        \end{CD}$$\\
By \cite[2.2]{A}, the induced morphism $\xi$ depends only on $f$. This completes the proof.
\end{proof}
 \section{Some properties of $\overline{\Tr}$ and $\overline{\lambda}$}
The goal of this section is to study the operations $\overline{\Tr}$ and $\overline{\lambda}$ on the category of $\mathcal{O}_X$-finitely presented sheaves of modules. Throughout this section, $\mathfrak{F}$ denotes a $\mathcal{O}_X$-finitely presented having $(t_1,t_2)$-rank, with $t_1 , t_2>0$.

First, note that any choice of free resolutions for $\mathfrak{F}$ has no effect, and both $\overline{\Tr}\,\mathfrak{F}$ and $\overline{\lambda}\,\mathfrak{F}$ are uniquely defined up to local stabilization. Indeed, let $\fp$ be an arbitrary point of a scheme $X$. By Definition~\ref{B1}, $\mathfrak{F}$ is generated by global sections, and the stalk $\mathfrak{F}_{\fp}$ is a finitely generated $(\mathcal{O}_X)_{\fp}$-module (for convenience, we write “$\mathcal{O}_{\fp}$–module’’). Hence, by \cite[2.5]{A}, both $\Tr(\mathfrak{F}_{\fp})$ and $\lambda(\mathfrak{F}_{\fp})$ are uniquely determined up to stabilization.
Also, considering stalks, we obtain the following commutative diagram:
\[
\begin{CD}
&&&&&&&\\
&&&& 0 @>>> (\mathfrak{F}^*)_{\fp} @>{(\varphi^*)_{\fp}}>> 
(\,\overset{t_1}{\oplus}\,\mathcal{O}_X^*)_{\fp} 
@>{(\phi^*)_{\fp}}>> 
(\,\overset{t_2}{\oplus}\,\mathcal{O}_X^*)_{\fp}
@>>> (\overline{\Tr}\,\mathfrak{F})_{\fp} @>>> 0\\
&&&&&& @VV{\cong}V @VV{\cong}V @VV{\cong}V \\
&&&& 0 @>>> (\mathfrak{F}_{\fp})^{+} @>>> 
\overset{t_1}{\oplus}\,(\mathcal{O}_{\fp})^{+} 
@>>> \overset{t_2}{\oplus}\,(\mathcal{O}_{\fp})^{+} 
@>>> \Tr(\mathfrak{F}_{\fp}) @>>> 0
\end{CD}
\]
and therefore, by the five Lemma, 
$(\overline{\Tr}\,\mathfrak{F})_{\fp} \cong \Tr(\mathfrak{F}_{\fp}).$ 
This implies $\overline{\Tr}(\mathfrak{F})$ is uniquely determined up to local stabilization. A similar argument shows that it is true for $\overline{\lambda}(\mathfrak{F})$.

It is well-known that, in the case where $X$ is connected, the rank of a locally free sheaf is the same everywhere. So, 
 for each $U \subseteq X$,
$(\overset{t}\oplus\mathcal{O}_X)\mid_U= \overset{t}\oplus(\mathcal{O}_X\mid_U).$ 
Therefore,
$\mathfrak{F}|_U$ is also $\mathcal{O}|_U$-finitely presented. However, by definition, the ranks of $\mathfrak{F}|_U$ may be smaller than those of $\mathfrak{F}$. For instance, suppose that $\mathfrak{F}$ is locally free and has $(t_1,t_2)$-rank. Then, for some an open subset $U\subseteq X$, $\mathfrak F|_U$ is free and has $(t,0)$-rank for some $t\leq t_1$. This raises the question of how these ranks behave under restriction to open subsets. The following theorem addresses this question. First,
set $$S=\{U \text{ is an open subset of } X:\mathfrak F\mid_U \text{ has }(t_1,t_2)\text{-rank}\}.$$ 
Note that $X\in S$ and so $S$ is nonempty.
\begin{thm}\label{B5}
Assume that $U\subseteq X$ is an open subset. Then the following statements hold:
\begin{itemize}
\item[(1) ]
If $U\in S$, then $(\overline{\lambda}\mathfrak{F})\mid_U\cong\;
\overline{\lambda}(\mathfrak{F}\mid_U)$
and
$(\overline{\Tr}\mathfrak{F})\mid_U\cong\;\overline{\Tr}(\mathfrak{F}\mid_U).$
\item[(2) ] If $\mathfrak F$ is coherent and $U$ is affine,
then $\underline{(\overline{\lambda}\mathfrak{F})\mid_U}\cong
\underline{\overline{\lambda}(\mathfrak{F}\mid_U)}$
and
$\underline{(\overline{\Tr}\mathfrak{F})\mid_U} \cong
\underline{\overline{\Tr}(\mathfrak{F}\mid_U)}.$
\end{itemize}
\end{thm}
\begin{proof}
Let $\overset{t_2}{\oplus}\mathcal{O}_X \rightarrow
\overset{t_1}{\oplus}\mathcal{O}_X \rightarrow
\mathfrak{F} \rightarrow 0$
be the free presentation of $\mathfrak{F}.$ 
Restricting to $U,$ we obtain
\begin{equation}\label{e.}
\overset{t_2}{\oplus}\mathcal{O}_X\mid_U \rightarrow
\overset{t_1}{\oplus}\mathcal{O}_X\mid_U \rightarrow
\mathfrak{F}\mid_U \rightarrow 0.
\end{equation}
\begin{itemize}
\item[(1) ]
By the assumption that $\mathfrak F|_U$ has $(t_1,t_2)$-rank. 
Applying $(-)^* := \mathcal{H}om_{\mathcal{O}_X\mid_U}(-,\mathcal{O}_X\mid_U)$ to \eqref{e.},
we yield the following commutative diagram
\[
\begin{CD}
(\overset{t_1}{\oplus}\mathcal{O}_X\mid_U)^*
    @>>> (\overset{t_2}{\oplus}\mathcal{O}_X\mid_U)^*
    @>>> \overline{\Tr}(\mathfrak{F}\mid_U))
    @>>> 0 \\
@VV{\cong}V @VV{\cong}V @VV{\exists}V\\
(\overset{t_1}{\oplus}\mathcal{O}_X)\mid_U
    @>>> (\overset{t_2}{\oplus}\mathcal{O}_X)\mid_U
    @>>> (\overline{\Tr}\mathfrak{F})\mid_U
    @>>> 0,
\end{CD}
\]
which implies that $(\overline{\Tr}\mathfrak{F})\mid_U\cong \overline{\Tr}(\mathfrak{F}\mid_U)$. The second assertion (concerning $\overline{\lambda}$) follows by an analogous argument, using one of the
exact sequences
$$0 \rightarrow  \overline{\lambda}\mathfrak{F}
\rightarrow \overset{t_2}{\oplus}\mathcal{O}_X
\rightarrow\overline{\Tr}\mathfrak{F}\rightarrow 0 \qquad\text{ or } \qquad 0 \rightarrow \mathfrak{F}^*
\rightarrow \overset{t_1}{\oplus}\mathcal{O}_X
\rightarrow\overline{\lambda}\mathfrak{F}\rightarrow 0.$$
\item[ (2) ] By the assumption,
there exist a ring $R$ and a finitely generated $R$-module $M$ such that $U = \Spec R$ and
$\mathfrak{F}\mid_{U} \cong \widetilde{M}$. Moreover, the presentation \eqref{e.} induces a free presentation $\overset{t_2}{\oplus}R \overset{\phi}{\rightarrow}
\overset{t_1}{\oplus}R \overset{\varphi}{\rightarrow}
M \rightarrow 0.$ By \cite[Propositions 2.5.5 and 2.5.2]{H} and \ref{B3}, we have $\underline{\coker \widetilde{\varphi^*}} \cong \underline{\overline{\lambda}(\mathfrak{F}\mid_{U})}$ and $\underline{\coker \widetilde{\phi^*}} \cong \underline{\overline{\Tr}(\mathfrak{F}\mid_{U})}.$ On the other hand, $\coker \widetilde{\varphi^*}= (\overline{\lambda}\mathfrak{F})\mid_U$ and $\coker \widetilde{\phi^*}= (\overline{\Tr}\mathfrak{F})\mid_U.$ The desired isomorphisms follow.
\end{itemize}
\end{proof}

The following corollary considers the case in which 
$\overline{\Tr}(\overline{\Tr}\,\mathfrak{F}) = \mathfrak{F}$.

\begin{cor}\label{B6}
Let $\mathfrak{F}$ be a coherent sheaf and $U$ be an affine open subset. Then, 
$\underline{\overline{\Tr}(\overline{\Tr}\,\mathfrak{F})}\mid_U
\cong \underline{\mathfrak{F}\mid_U}.$ 
Moreover, suppose that $X$ admits an affine open cover $\{U_i\}_{i\in I}$ such that $U_i\in S$ for every $i.$ Then $\overline{\Tr}(\overline{\Tr}\,\mathfrak{F}) \cong \mathfrak{F}$.
\end{cor}

\begin{proof}
By \cite[Proposition 2,5.4]{H}, 
there exist a ring $R$ and a finitely generated $R$-module $M$ such that $U = \Spec R$ and
$\mathfrak{F}\mid_{U} \cong \widetilde{M}$. The first assertion now follows from \ref{B5}(2) and \ref{B3}(2 and 5):
\[
\underline{(\overline{\Tr}(\overline{\Tr}\,\mathfrak{F}))\mid_{U}}
\cong \underline{\overline{\Tr}(\overline{\Tr}(\mathfrak{F}\mid_{U}))}
\cong \underline{\mathfrak{F}\mid_{U}.}
\]
Similarly, \ref{B5}(1) and \ref{B3}(5) imply that $
(\overline{\Tr}(\overline{\Tr}\,\mathfrak{F}))\mid_{U_i}
\cong \mathfrak{F}\mid_{U_i}.$
 Taking stalks yields the second assertion.
\end{proof}

In view of the definition of $\overline{\Tr}\,\mathfrak{F}$, it is natural to ask whether
$\overline{\Tr}\,\mathfrak{F}$ vanishes. To this end, 
we introduce the notion of an 
$\mathcal{O}_X$-projective sheaf.

\begin{defn}\label{B22}
We say that $\mathfrak{F}$ is \emph{$\mathcal{O}_X$-projective} if 
$\mathfrak{F}$ is a direct summand of a free $\mathcal{O}_X$-modules of finite rank.  
In other words, there exist an $\mathcal{O}_X$-modules $\mathfrak{h}$ and a free 
$\mathcal{O}_X$-modules $\mathfrak{G}$ of finite rank such that
$\mathfrak{G} = \mathfrak{F} \oplus \mathfrak{h}.$
\end{defn}

The next lemma shows the relationship between the $\mathcal{O}_X$-projectiveness of a coherent sheaf and the projectiveness of its global section.

\begin{lem}\label{B23}
Let $X = \Spec R$ be an affine scheme and $\mathfrak{F}$ be a coherent sheaf.
Then $\mathfrak{F}$ is $\mathcal{O}_X$-projective if and only if $\Gamma(X,\mathfrak{F})$ is projective.
\end{lem}

In the remainder of this section, we classify locally free sheaves in terms of their transpose.

\begin{prop}\label{B4}
Let $\mathfrak{F}$ be a coherent sheaf. Then the following statements are equivalent:
\begin{itemize}
    \item[(1)] $\mathfrak{F}$ is locally free.
    \item[(2)] $\underline{\overline{\Tr}\,\mathfrak{F}} = 0$.
    \item[(3)] $\mathfrak{F}$ is $\mathcal{O}_X$-projective.
\end{itemize}
\end{prop}

\begin{proof}
We may assume that $\mathfrak{F}$ is stable.

\item[\textbf{(1 $\Rightarrow$ 2)}]
Let $U$ be an arbitrary affine open subset.  
By the assumption together with \ref{B21} and \ref{B5}, we have
$\overline{\Tr}(\mathfrak{F}\mid_U)=0$ and so $\overline{\Tr}\,\mathfrak{F} = 0$.

\item[\textbf{(2 $\Rightarrow$ 3)}]
By assumption, $X$ can be covered by affine open subsets 
$U_i = \Spec R_i$ such that for each $i$ there exists a finitely generated 
$R_i$-module $M_i$ with
$\mathfrak{F}\mid_{U_i} \cong \widetilde{M_i}.$ By \ref{B3}, $\Tr M_i = 0$.  
Thus, \cite[2.6]{A} implies that each $M_i$ is projective.  
Therefore, there exist an $\mathcal{O}(U_i)$-module $N_i$ and an integer $l_i$ such that
$\oplus^{l_i} \mathcal{O}(U_i) = M_i \oplus N_i.$
Taking $l := \max\{l_i\}$ and replacing the $N_i$ by suitable module, we may assume that
$\oplus^l\mathcal{O}(U_i) = M_i \oplus N_i,$ for all $i.$
Applying the sheafification functor $\widetilde{(-)}$, we obtain
$\oplus^ l\mathcal{O}_X\mid_{U_i}= \mathfrak{F}\mid_{U_i} \oplus \widetilde{N_i}.$
Thus $\mathfrak{F}\mid_{U_i}$ is $\mathcal{O}_X\mid_{U_i}$-projective.  
Set $\mathfrak{G}_i := \widetilde{N_i}$.  
For $i \neq j$ we have
$\oplus ^l\mathcal{O}_X\mid_{U_i \cap U_j}
= \mathfrak{F}\mid_{U_i \cap U_j} \oplus \mathfrak{G}_i\mid_{U_i \cap U_j}
= \mathfrak{F}\mid_{U_i \cap U_j} \oplus \mathfrak{G}_j\mid_{U_i \cap U_j},$
so
$\mathfrak{G}_i\mid_{U_i \cap U_j} \cong \mathfrak{G}_j\mid_{U_i \cap U_j}.$
By the gluing lemma for sheaves, there exists a unique sheaf $\mathfrak{G}$ such that
$\mathfrak{G}\mid_{U_i} = \mathfrak{G}_i$ for all $i.$
Hence,
$\oplus^ l\mathcal{O}_X\mid_{U_i}
= \mathfrak{F}\mid_{U_i} \oplus \mathfrak{G}\mid_{U_i}$
This shows the result.

\item[\textbf{(3 $\Rightarrow$ 1)}]
Let $\fp$ be a point of $X$, and assume that $\mathfrak{F}$ is $\mathcal{O}_X$-projective.  
Then there exist an affine open subset $U \subseteq X$ containing $\fp$ and an 
$\mathcal{O}(U)$-module $M$ such that
$\mathfrak{F}\mid_U \cong \widetilde{M}.$
By assumption, there exist an $\mathcal{O}_X$-modules $\mathfrak{h}$ and an integer $l$ such that
$\oplus^ l \mathcal{O}_X= \mathfrak{F} \oplus \mathfrak{h}.$
Applying $(-)_{\fp}$ and using Lemma ~\ref{B23}, this implies that $M_\fp$ is projective and hence is free.
The claim now follows from \cite[Exercise 2.5.7(b)]{H}.
\end{proof}

\section{linkage of sheaves of modules}
In this section, we first introduce the concept of the linkage of sheaves of modules and
study some of their basic properties. Then, using these properties, we show that the sheaf
obtained from gluing schemes and gluing linked sheaves of modules is itself linked.

\begin{defn}\label{B.2}
Let $\mathfrak{F}$ and $\mathfrak{G}$ be two sheaves.
We say that $\mathfrak{F}$ and $\mathfrak{G}$ are linked, denoted by
$\mathfrak{F} \sim \mathfrak{G}$, if
$\mathfrak{F} \cong \overline{\lambda}\,\mathfrak{G}$ and 
$\mathfrak{G} \cong \overline{\lambda}\,\mathfrak{F}$.
A sheaf $\mathfrak{F}$ is called a linked sheaf if there exists a
sheaf 
$\mathfrak{G}$ such that
$\mathfrak{F} \sim \mathfrak{G}$. In particular,
$(\overline{\lambda})^{2}\mathfrak{F} \cong \mathfrak{F}$.
\end{defn}

\begin{note}
Via \ref{B21}(iii), a free sheaf of modules $\mathfrak{F}$ cannot be linked to any other sheaf of modules unless $\mathfrak{F}=0$. Hence, throughout the remainder of this paper, all sheaves are assumed to be nonzero and non-free.
\end{note}
In view of the definition, if $\mathfrak{F}$ is a linked sheaf then
$\overline{\lambda}\,\mathfrak{F}$ is also linked. Indeed,
$\overline{\lambda}\,\mathfrak{F} \sim \mathfrak{F}$.
Moreover, if $\mathfrak{F}$ is a coherent sheaf, then
$\overline{\lambda}\,\mathfrak{F}$ is a linked coherent sheaf.

The following theorem considers the case where $X=\Spec(R)$ is affine and shows that
linkedness of an $R$-module is equivalent to the linkedness of its associated sheaf.

\begin{thm}\label{C1}
Let $X=\Spec(R)$ be an affine scheme and let $M$ and $N$ be finitely generated $R$-modules.  
The following statements are equivalent:
\begin{enumerate}
\item $M \sim N$.
\item $\widetilde{M} \sim \widetilde{N}$.
\item $\Gamma(U,\widetilde{M}) \sim \Gamma(U,\widetilde{N})$ for every open subset $U\subseteq X$  such that $\widetilde M|_U$ and $\widetilde N|_U$ are nonfree, and $\mu(\widetilde{M}\mid_U)= \mu (M)$ and $\mu(\widetilde{N}\mid_U)= \mu (N)$.
\item $\Gamma(X,\widetilde{M}) \sim \Gamma(X,\widetilde{N})$.
\end{enumerate}
\end{thm}

\begin{proof}
First, note that, by \cite[Proposition 3,3]{MS}, a linked module is stable. The results follow immediately from an argument similar to the one used in the proof of \ref{B5}(1), together with \ref{B3}.
\end{proof}
\begin{cor}\label{D}
Let $X$ is an affine scheme and $\mathfrak{F}$ and $\mathfrak{G}$ be coherent sheaves.
Then
$\mathfrak{F} \sim \mathfrak{G}$ if and only if 
$\Gamma(X,\mathfrak{F}) \sim \Gamma(X,\mathfrak{G}).$
\end{cor}
One may obtain the following examples.
\begin{exam}
\begin{itemize}
\item[( i )]
Let $R = k[x,y]/(xy)$ where $k$ is a field, and let $M = (x)$ be the ideal generated by $x$. Then $M$ is not a free $R$–module, but it sits in a short exact sequence
\[
0 \longrightarrow R \xrightarrow{\cdot y} R \longrightarrow M \longrightarrow 0.
\]
Passing to the associated sheaf $\mathfrak{F} = \widetilde{M}$ on $X = \Spec R$, we obtain a coherent $\mathcal{O}_X$–modules with a free resolution
\begin{equation}\label{e''}
0 \rightarrow \mathcal{O}_X \xrightarrow{\cdot y} \mathcal{O}_X \rightarrow \mathfrak{F} \rightarrow 0.
\end{equation}
The corresponding sequence (\ref{e''}) gives
\[
0 \longrightarrow \mathfrak{F}^\ast \longrightarrow \mathcal{O}_X 
\xrightarrow{\cdot y} \mathcal{O}_X \longrightarrow \overline{\Tr} \mathfrak{F} \longrightarrow 0,
\]
where the middle morphism is given by multiplication by $y$; that is, for any open subset $U \subseteq X$, it is the morphism
\[
\mathcal{O}_X(U) \longrightarrow \mathcal{O}_X(U),
\qquad f \longmapsto yf.
\] Thus, $\overline{\Tr}\mathfrak{F} \cong \mathcal{O}_X/(y)\mathcal{O}_X$ and $\overline{\lambda} \mathfrak{F} \cong \mathcal{O}_X/(x)\mathcal{O}_X$. Therefore, $\mathfrak{F}$ and $\overline{\lambda} \mathfrak{F}$ are linked, corresponding to the classical linkage of the ideals $(x)$ and $(y)$ in $R$. Indeed, $0:_R(x)= (y)$ and $0:_R(y)=(x).$ Thus, via \cite[Proposition 2.1]{MS}, $\lambda M \cong R/(y)$ and $\lambda (R/(y)) \cong M$.

\item[( ii )]
Let $X = \Spec R$ where $R = k[x,y,z]/(xz, yz)$. Consider the $R$–module $M = (x,y)$. Then the associated sheaf $\mathfrak{F} = \widetilde{M}$ has the minimal free resolution
\[
0 \longrightarrow R^2 
\xrightarrow{
\begin{pmatrix}
z & 0 \\
0 & z
\end{pmatrix}}
R^2 \longrightarrow M \longrightarrow 0.
\]
From this resolution, we obtain
\[
\overline{\Tr} \mathfrak{F} \cong \widetilde{R/(z)}, \qquad 
\overline{\lambda} \mathfrak{F} \cong \widetilde{R/(x,y)}.
\]
Thus, $\mathfrak{F}$ and $\overline{\lambda} \mathfrak{F}$ are linked. This demonstrates that linkage of sheaves reflects geometric relationships between components of the scheme defined by intersecting subvarieties.

\end{itemize}
\end{exam}




To study linked sheaves more effectively,
we introduce the following notion.

\begin{defn}
A sheaf of $\mathcal{O}_X$-modules $\mathfrak{F}$ is called a \emph{syzygy} if it can be embedded
into a free sheaf of finite rank.
\end{defn}
\begin{cor}\label{C5}
Then the following statements hold:
\begin{itemize}
\item[ (i) ] 
If $\mathfrak{F}$ be a linked sheaf of modules then $\mathfrak{F}$ is a syzygy. In particular, for every open subset $U\subseteq X$,
the module $\mathfrak{F}(U)$ is a syzygy and therefore $1$-torsion free. 
\item[ (ii) ] 
Moreover, in the case where $X$ is affine and $\mathfrak{F}$ is coherent, $\mathfrak{F}$ is linked if and only if $\mathfrak{F}$ is syzygy and stable.
\end{itemize}
\end{cor}
\begin{proof}
\begin{itemize}
\item[ (i) ] 
By definition, there exists a sheaf of modules $\mathfrak{G}$ such that
\[
0\rightarrow\mathfrak{F} \longrightarrow
\left(\overset{l}{\oplus}\mathcal{O}_X\right)^*
\longrightarrow \overline{\Tr}\mathfrak{G}.
\]
Thus, $\mathfrak{F}$ is a syzygy. In particular, since $\Gamma(U,-)$ is a left-exact functor, it follows that the module $\mathfrak{F}(U)$ is a syzygy and, therefore, $1$-torsion free.
\item[ (ii) ] 
The final result follows from \ref{D} and \cite[Corollary 5.6]{MS}.
\end{itemize}
\end{proof}
The next result examine whether properties of the ring affect the linkedness of a module
and the linkedness of a sheaf of modules.

\begin{exam}\label{A.3}
Let $R$ be a PID and let $X=\Spec(R)$.  
Then there exists no linked coherent sheaf of $\mathcal{O}_X$-modules.  
Indeed, if $\mathfrak{F}$ were linked, then by \ref{C5}, $\mathfrak{F}(U)$ would be a syzygy
for each $U\subseteq X$, hence a free module.  
Therefore $\mathfrak{F}(U)$ cannot be a linked module, contradicting Theorem~\ref{C1}.
\end{exam}
 The next theorem shows that linkedness of a sheaf is a local property.

\begin{thm}\label{C2.3}
$\mathfrak{F}$ is linked if and only if $\mathfrak{F}\!\mid_U$ is linked as an 
$\mathcal{O}_X\!\mid_U$-modules for every open subset $U\subseteq X$ with $U\in S$.
\end{thm}
\begin{proof}

Assume that $\mathfrak{F}$ is linked and 
$U\in S$.  
Since $\mathfrak{F}\cong \overline{\lambda}^{\,2}\mathfrak{F}$, Theorem~\ref{B5}(1) gives
\[
\mathfrak{F}\!\mid_U
\cong (\overline{\lambda}^{\,2}\mathfrak{F})\!\mid_U
\cong \overline{\lambda}\bigl((\overline{\lambda}\mathfrak{F})\!\mid_U\bigr)
\cong \overline{\lambda}\bigl(\overline{\lambda}(\mathfrak{F}\!\mid_U)\bigr).
\]
Hence $\mathfrak{F}\!\mid_U$ is linked. The converse follows by taking $U:=X.$ 
\end{proof}




Is a sheaf obtained by gluing linked sheaves again linked? Two next theorems show that this is true in some cases. First,
we introduce the following notion.

\begin{defn}
Let $\mathfrak{F}$ and $\mathfrak{G}$ be 
sheaves having the same $(t_1,t_2)$-rank.  
We say that $\mathfrak{F}$ and $\mathfrak{G}$ are \emph{$(t_1,t_2)$-co-rank}.
\end{defn}

\begin{thm}\label{C4}
Let \(t_1\) and \(t_2\) be positive integers, and let \(\{(U_i,\mathfrak F_i)\}\) be a family consisting of 
open subsets \(U_i\subseteq X\) and 
 \((t_1,t_2)\)-co-rank linked sheaves \(\mathfrak F_i\) satisfying the gluing lemma. Let \(\mathfrak F\) be the sheaf obtained by gluing, and suppose that \(\mathfrak F\) is $\mathcal{O_X}$-finitely presented.
Then 
$\mathfrak{F}$ is a linked sheaf.
\end{thm}

\begin{proof}
We claim that $\mathfrak{F}$ has $(t_1,t_2)$-rank. Indeed, 
suppose that $\mathfrak F$ has $(k_1,k_2)$-rank and let
$\fp\in X$ be an arbitrary point. Then, there exist an open neighborhood $U_i$ of $\mathfrak{p}$ and a free presentation $\overset{t_2}{\oplus}\mathcal{O}_{\fp} 
\rightarrow\overset{t_1}{\oplus}\mathcal{O}_{\fp} 
\rightarrow\mathfrak{F}_{\fp}\rightarrow 0$. Thus, $\mathfrak{F}_{\fp}$ can be generated by at most $t_1$ elements and $\Ker(\overset{t_1}{\oplus}\mathcal{O}_{\fp} 
\rightarrow\mathfrak{F}_{\fp})$ can be generated by at most $t_2$ elements. Therefore, via \cite[ Definition 2,5.16.2]{H}, there exists a free presentation $\overset{t_2}{\oplus}\mathcal{O} 
\rightarrow\overset{t_1}{\oplus}\mathcal{O} 
\rightarrow\mathfrak{F}\rightarrow 0$. Consequently, $(k_1,k_2)\leq(t_1,t_2)$ with respect to the lexicographic order. On the other hand, restricting the corresponding global free presentation $\overset{k_2}{\oplus}\mathcal{O} 
\rightarrow\overset{k_1}{\oplus}\mathcal{O} 
\rightarrow\mathfrak{F}\rightarrow 0$ to each $U_i$ yields a free presentation of $(\mathfrak F|{U_i}).$ Hence, we obtain $(k_1,k_2)\geq(t_1,t_2).$ Combining these inequalities, we conclude that $(k_1,k_2)=(t_1,t_2)$.
Furthermore, in view of \cite[Exercise 2,1.22]{H} and \ref{B5}, we have
\[
\mathfrak{F}\!\mid_{U_i}
\cong \mathfrak{F}_i
\cong \overline{\lambda}(\overline{\lambda}\mathfrak{F}_i)
\cong \overline{\lambda}\bigl(\overline{\lambda}(\mathfrak{F}\!\mid_{U_i})\bigr)
\cong (\overline{\lambda}(\overline{\lambda}\mathfrak{F}))\!\mid_{U_i}.
\]
The definition of a sheaf yields $\mathfrak{F}\cong \overline{\lambda}(\overline{\lambda}\mathfrak{F})$,  
so $\mathfrak{F}$ is linked.
\end{proof} 
The following corollary is a generalization of the preceding theorem. Instead of considering a single scheme, we consider several subschemes.

\begin{cor}\label{C}
\textbf{Gluing of linked sheaves.}  
Let $(X,\mathfrak{F})$ be a sheaf obtained by gluing a family $\{(X_i,\mathfrak{F}_i)\}$  
under the gluing conditions for schemes and sheaves of modules, \cite[Exercise 2,2.12]{H}.
Also, assume that each $X_i$ is connected.  
Then $\mathfrak{F}$ is linked provided that there exist positive integers $t,k$ such that  
$\{(X_i,\mathfrak{F}_i)\}$ are $(t,k)$-co-rank linked sheaves and \(\mathfrak F\) is $\mathcal{O_X}$-finitely presented.
\end{cor}

\begin{proof}
Gluing connected schemes along nonempty intersections yields a connected scheme $X$.  
(For example, $X = (X_1\cup X_2)/{\sim}$ remains connected because $X_1\cap X_2\neq\varnothing$.)  
Thus the conclusion follows from Theorem~\ref{C4}.
\end{proof}

The next result concerns the existence of linked subsheaves.

\begin{prop}\label{L1}
Let $U\subseteq X$ be an affine open subset such that $\mathcal{O}(U)$ is not an integral domain.  
Then $\mathfrak{F}\!\mid_U$ has a nonzero linked subsheaf if and only if 
\[
\Ass(\mathfrak{F}(U))\cap \Ass(\mathcal{O}(U))\neq \varnothing.
\]
\end{prop}

\begin{proof}
Suppose $\fp\in \Ass(\mathfrak{F}(U))\cap \Ass(\mathcal{O}(U))$.  
In view of \cite[2.5]{JS2}, $\fp\neq 0$ is a linked ideal. Moreover, using \cite[Proposition 2,1]{MS},  
$\mathcal{O}(U)/\fp$ is a linked submodule of $\mathfrak{F}(U)$.  
Hence, by Theorem~\ref{C1}, 
$(\mathcal{O}(U)/\fp)^{\sim}$
is a nonzero linked subsheaf of $\mathfrak{F}\!\mid_U$. 

Conversely, assume $\mathfrak{F}\!\mid_U$ contains a nonzero linked subsheaf $\mathfrak{F}'$. Since $U$ is affine
and $\mathfrak{F}'$ is nonzero, 
we have $\mathfrak{F}'(U)\neq0$ and so its set of associated primes is nonempty.
Now, by definition, 
$\Ass(\mathfrak{F}'(U))\subseteq \Ass(\mathcal{O}(U))\cap \Ass(\mathfrak{F}(U))$
as desired.
\end{proof}

\section{linkage of Sheaves on Projective schemes}

In this section, we compare invariants of cohomology modules and the regularity of certain linked sheaves of modules.

Now assume $S= \oplus_{i=0} S_i$ is a graded ring and $X=\proj(S)$.  
For an $\mathcal{O}_X$-modules $\mathfrak{F}$, define the graded cohomology modules
\[
H^i_*(X,\mathfrak{F}) := \bigoplus_{d\in\mathbb{Z}} H^i(X,\mathfrak{F}(d)).
\]
It is well known that, if $S$ is generated by $S_1$ over $S_0$ and $\mathfrak{F}$ is a coherent $\mathcal{O}_X$-modules, then
\begin{equation}\label{e'.1}
\widetilde{H^0_*(X,\mathfrak{F})} \cong \mathfrak{F},
\end{equation}
(see \cite[2, 5.15]{H}). However, for a finitely generated graded $S$-module $M$, the $S$-module $H^0_*(X,\widetilde{M})$ is not necessarily finitely generated, nor is it necessarily isomorphic to $M$.
From now on, assume that $X$ is a projective scheme over a Noetherian ring $S_0$, equipped with a very ample invertible sheaf $\mathcal{O}_X(1)$, and
the sheaves are coherent on $X$ unless stated otherwise. 

\begin{rem}\label{A3}
(A corollary of Serre's theorem.)
\begin{itemize}
\item[(i)]  
There is a closed immersion $i: X\rightarrow\mathbb{P}^r_{S_0}= \proj S_0[x_0,...,x_r]$ of schemes over $S_0$, for some $r,$ such that $\mathcal{O}_X(1)= i^* \mathcal{O}_{\mathbb{P}^r_{S_0}}(1).$
\item[(ii)]  
$\mathfrak{F}$ admits a invertible resolution of the form  
\[
\cdots \rightarrow \bigoplus_{j=1}^{a_i}\mathcal{O}_X(a_{ij}) 
\rightarrow \cdots
\rightarrow \bigoplus_{j=1}^{a_2}\mathcal{O}_X(a_{1j})
\rightarrow \bigoplus_{j=1}^{a_1}\mathcal{O}_X(a_{0j})
\rightarrow \mathfrak{F} \rightarrow 0.
\]

\item[(iii)]
$H^i(X,\mathfrak{F}(d))$ is a finitely generated $S_0$-module for all $i\ge 0$.
\end{itemize}
\end{rem}
\begin{con}
From now on, for projective sheaves, we use invertible resolutions instead of free resolutions, and define the operators $\overline{\Tr}$ and $\overline{\lambda}$ analogously to those in \ref{B1}.
\end{con}
The following two lemmas and corollary establish properties of these functors on a projective scheme.
\begin{lem}\label{B'3}
Let $M$ be a stable finitely generated graded $S$-module. Then $\underline{\overline{\lambda}\,\widetilde{M}} \cong \underline{\widetilde{\lambda M}}$ and $\underline{\overline{\Tr}\,\widetilde{M}} \cong \underline{\widetilde{\Tr M}}$.
\end{lem}
The following corollary computes certain cohomology groups in a special case. Recall that an $S$-module $N$ is saturated if $H^0_*(X,\widetilde{N})\cong N$.
\begin{cor}\label{b'3}
Let $M$ be a finitely generated graded $S$-module such that $\lambda M$ is saturated and the integer $r$ appearing in \ref{A3} satisfies $r\geq 2$.
Then
$H^0_*(X,(\widetilde{M})^*)\cong \Hom_S(M,S)$
and
$H^1_*(X,(\widetilde{M})^*)=0.$
\end{cor}
\begin{proof}
For convenience, set $\mathfrak{F}=\widetilde{M}$. Note that $i_*\mathfrak{F}$ is coherent and, by \cite[Exercise 2,5.5]{H}, $i_*\mathfrak{F}$ is coherent on $\mathbb{P}^r_{S_0}.$ Moreover, the cohomology groups of $i_*\mathfrak{F}$ and $\mathfrak{F}$ are isomorphic. Hence, we may assume that $X=\mathbb{P}^r_{S_0}$.
 Suppose that 
\begin{equation}\label{e'.3}
\bigoplus_{j=1}^{a_1} S(-a_{1j})\xrightarrow{g} M\longrightarrow 0.
\end{equation}
is a minimal cover of $M.$ Applying the functor $\widetilde{(-)}$ and then the functor $(-)^*$, we obtain the exact sequences
$\bigoplus_{j=1}^{a_1} \mathcal{O}_X(-a_{1j})\longrightarrow \mathfrak{F}\longrightarrow 0,$
and 
\begin{equation}\label{e'.2}
0 \longrightarrow \mathfrak{F}^*
\longrightarrow \bigoplus_{j=1}^{a_1}\mathcal{O}_X(a_{1j})
\longrightarrow \overline{\lambda}\mathfrak{F}
\longrightarrow 0.
\end{equation}
Sheafification of the sequence in \eqref{e'.2}, together with \cite[Theorem 3,5.1]{H}, yields the exact sequence
\begin{equation}\label{B.3}
0 \rightarrow H^0_*(X,\mathfrak{F}^*)
\rightarrow \bigoplus_{j=1}^{a_1}S(a_{1j})\xrightarrow{h} H^0_*(X,\overline{\lambda}\widetilde{M})
\rightarrow H^1_*(X,\mathfrak{F}^*)
\rightarrow 0.
\end{equation}
Here $h$ is a natural morphism and, by the assumption and \ref{B'3}, $H^0_*(X,\overline{\lambda}\widetilde{M})=\lambda M$. 

On the other hand, applying $(-)^+ := \Hom_S(-,S)$ to \eqref{e'.3},
we have the exact sequence
$$
0 \longrightarrow M^+
\xrightarrow{g^+} \bigoplus_{j=1}^{a_1} S(a_{1j})
\xrightarrow{h} \lambda M
\longrightarrow 0.
$$
Comparing the last two exact sequences and applying the Five Lemma, we conclude the desired results.
\end{proof}
\begin{lem}

Let 
$k$ be an integer.  
Then
$\overline{\Tr}(\mathfrak{F}(k)) \cong (\overline{\Tr}\mathfrak{F})(-k)$ and $\overline{\lambda}(\mathfrak{F}(k)) \cong (\overline{\lambda}\mathfrak{F})(-k).$
\end{lem}
\begin{proof}
It is straightforward by definition.
\end{proof}
The above relation shows that
$\overline{\lambda}^{\,2}(\mathfrak{F}(k))\cong
(\overline{\lambda}^{\,2}\mathfrak{F})(k).$
Hence, the definition of linkage for sheaves on projective schemes is similar to the preceding definition; namely, $\mathfrak{F}$ is linked if and only if
$\mathfrak{F}\cong\overline{\lambda}^{\,2}\mathfrak{F}.$ Therefore, all the concepts and most of the results of the preceding chapter carry over to this setting, with analogous proofs.

\begin{exam}
Let $X = \mathbb{P}^2_k$ be the projective plane over a field $k$ and let $\mathfrak{F} = \mathcal{O}_X(1)$. Since $\mathfrak{F}$ is invertible of rank one, 
both $\overline{\Tr} \mathfrak{F}$ and $\overline{\lambda} \mathfrak{F}$ vanish. Hence $\mathfrak{F}$ is not linked. 
\end{exam}
\begin{thm}\label{D3}
Then the following statements are equivalent.
\begin{itemize}
\item[(i) ]
$\mathfrak{F}$ and $\mathfrak{G}$ are linked and $\lambda(H^0_*(X,\mathfrak{F}))$ and $\lambda(H^0_*(X,\mathfrak{G}))$ are saturated.
\item[(ii) ]
$i_*\mathfrak{F}$ and $i_*\mathfrak{G}$ are linked and $\lambda(H^0_*(\mathbb{P}^r_{S_0},i_*\mathfrak{F}))$ and $\lambda(H^0_*(\mathbb{P}^r_{S_0},i_*\mathfrak{G}))$ are saturated.
\item[(iii) ]
$H^0_*(X,\mathfrak{F})$ and $H^0_*(X,\mathfrak{G})$ are linked.
\item[(iv) ]
$H^0_*(\mathbb{P}^r_{S_0},i_*\mathfrak{F})$ and $H^0_*(\mathbb{P}^r_{S_0},i_*\mathfrak{G})$ are linked.
\end{itemize}
\end{thm}
\begin{proof}
First note that, by the proof of \ref{b'3}, we may reduce to the case $X=\mathbb{P}^r_{S_0}$ and, it is sufficient to prove $(i) \Longleftrightarrow (iii).$
For convenience, set $M:=H^0_*(X,\mathfrak{F})$ and  
$N:=H^0_*(X,\mathfrak{G})$. Recall that, by \eqref{e'.1}, both $M$ and $N$ are saturated.

Suppose that
$\mathfrak{F}\sim \mathfrak{G}.$ Thus, first, $M$ is a syzygy and hence a finitely presented $S$-module and, second, $\mathfrak{F}\cong \overline{\lambda}\mathfrak{G}.$ It follows that $H^0_*(X,\mathfrak{F})\cong H^0_*(X,\overline{\lambda}\mathfrak{G})$, and, together with the assumption and \ref{B'3}, this gives $M\cong \lambda N$. Similarly, $N\cong \lambda M$.

The converse follows directly by applying the twist functor and using \ref{B'3} together with (\ref{e'.1}). Moreover, the second part follows from 
the isomorphisms 
$M\cong \lambda N$ and $N\cong \lambda M$.
\end{proof}
Let 
$X$ be a projective scheme over $k$ and $\mathcal{O}_X(1)$ be a very ample invertible sheaf on $X.$ 
Then,
a polynomial $P_{\mathfrak{F}}(z)\in \mathbb{Q}[z]$ exists such that, for all $n\in \mathbb{Z}$, $\chi(\mathfrak{F} (n))= P_{\mathfrak{F}}(n)$ where $\chi (\mathfrak{F})$ is the Euler charactistic of $\mathfrak{F}$. $P_{\mathfrak{F}}(z)$ is called the Hilbert polynomial of $\mathfrak{F}$ with respect to the sheaf $\mathcal{O}_X(1)$.
If $\dim H^0_*(X,\mathfrak{F})>0$, then its Hilbert polynomial can be written in the form $$P_{\mathfrak{F}}(z)= h_0(\mathfrak{F}) \begin{pmatrix}
z  \\
d-1
\end{pmatrix}+h_1(\mathfrak{F}) \begin{pmatrix}
z  \\
d-2
\end{pmatrix}+ ... + h_{d-1}(\mathfrak{F})$$ 
where $h_i(\mathfrak{F})$ is integer, for all $i.$

On the other hand, the degree and the index of regularity of a graded module $N$ are defined by the coefficients of $h_i(N)$. If $\dim N>0$ then $h_0(N)$ is called the degree of $N$, $\deg (N)$. Otherwise the degree of $N$ is the length of $N$, $l(N).$ Also, the index of regularity of $N$, denoted by $r(N)$, is defined as $$r(N)=inf \{i \in \mathbb{Z}\mid h_N(j)=P_N(j) \text{ for all } j\geq i \},$$
where $h_N(j)=rank_k[N]_j$ is the Hilbert function of the graded $R$-module $N$.
Inspired by these definitions and by the fact that the Hilbert polynomial of $\mathfrak{F}$ with respect to the sheaf $\mathcal{O}_X(1)$ coincides with the Hilbert polynomial of $M:= H^0_*(X,\mathfrak{F}),$ we define the degree and index of regularity of $\mathfrak{F}$ as follows:
\begin{defn}
The degree of $\mathfrak{F}$, denoted by $\deg (\mathfrak{F}),$ is defined as $(\dim \Supp \mathfrak{F})!$ times the degree of $M,$ and the index of regularity of $\mathfrak{F}$, denoted by $r(\mathfrak{F})$, is taken to be the index of regularity of $M.$
\end{defn}
Nagel gave a definition of the linkage of modules over Gorenstein rings, denote this relation by $\sim_{C}$, where $C$ is a fix module. In his work, he pointed out that Martsinkovsky and Stocker's definition is a special case of his concept of linkage of modules. Therefore, the results that Nagel obtained about the degree and the index of regularity of linked module can be generalized to sheaves of modules. The following corollary states this.
\begin{cor}\label{D2}
Let $X=\mathbb{P}^n_k$ and
$\mathfrak{F}\sim\mathfrak{G}$ such that $\lambda(H^0_*(X,\mathfrak{F}))$ and $\lambda(H^0_*(X,\mathfrak{G}))$ are saturated. Assume that $...
\rightarrow \bigoplus_{j=1}^{a_0}\mathcal{O}_X(-a_{0j})
 \rightarrow \mathfrak{F} \rightarrow 0$ is a invertible resolution of $\mathfrak{F}$ and set $s:=\frac{\min\{a_{0j}\}+\max\{a_{0j}\}}{2}.$
Then, the following statements hold.
\begin{itemize}
\item[ (i) ]
$\deg (\mathfrak{G})=a_0 - \deg (\mathfrak{F}),$
and if in addition $n\geq 1,$ then 
$$h_1(\mathfrak{G})= \frac{s-n}{n!}(\deg \mathfrak{F}- \deg \mathfrak{G})+ h_1(\mathfrak{F})$$
\item[ (ii) ] If $H^*_0(X,\mathfrak{F})$ is locally Cohen-Macaulay then
$$P_{\mathfrak{F}}(z)=P_{\mathcal{O}_X}(z)- (-1)^n P_{\mathfrak{G}}(2s-n-1-z).$$
\end{itemize}
\end{cor}
\begin{proof}
Let $M:=H^0_*(X,\mathfrak{F})$, 
$N:=H^0_*(X,\mathfrak{G})$ and
$S:=k[x_0,\ldots,x_n].$ 
By \ref{D3} and \cite[Remark 3.20]{N}, we have $M\sim_C N$, where $C:=\bigoplus_{j=1}^{a_0}S(-a_{0j}).$ Consequently, $\Ass M \cup\Ass N \subseteq \Ass S$. It follows that $\dim_S M =\dim_S N =n+1$ and $\dim \Supp \mathfrak{F} =\dim \Supp \mathfrak{G} = n.$
We then have
$\deg\mathfrak{F}=n!\deg M.$ For any graded module $L$, define
$a(L):=\min\{i\mid L_i\neq0\}.$
Clearly, $a(C)=\min\{a(S)+a_{0j}\}=\min\{a_{0j}\},$
while $r(C)=\max\{r(S)+a_{0j}\}=-n+\max\{a_{0j}\}.$ Consequently, $r(C)+a(C)-1=2s-n-1.$
Moreover, for any integer $i$, the leading coefficient of $S(i)$ is $1/n!.$ Thus,
$\deg C=\frac{a_0}{n!}.$ 
The desired results now follow from the definition and \cite[4.3]{N}.
\end{proof}
Let 
$S = k[x_0,\dots,x_n]$.
$\mathfrak{F}$ is called $m$-regular if
$H^i(\mathbb{P}^n, \mathfrak{F}(m-i)) = 0$, for all $i \in \mathbb{N}.$
If $\mathfrak{F}$ is $m$-regular, then it is also $(m+1)$-regular.  So,
the regularity of $\mathfrak{F}$ is defined by
\[
\reg(\mathfrak{F}) := \min \{ m \mid \mathfrak{F} \text{ is } m\text{-regular} \}.
\]
\begin{note}
There is a relation between regularity and minimal free resolutions. Let
$M := H^0_*(\mathbb{P}^n_k, \mathfrak{F}),$
and assume that $M$ is finitely generated. Let 
\[
0 \rightarrow \dots \rightarrow \bigoplus_{j=1}^{a_i} S(-a_{ij}) \rightarrow \dots \rightarrow \bigoplus_{j=1}^{a_1} S(-a_{1j}) \rightarrow \bigoplus_{j=1}^{a_0} S(-a_{0j}) \rightarrow M \rightarrow 0
\]
be a minimal free resolution of $M.$ Sheafifying gives a invertible resolution of $\mathfrak{F}$:
\begin{equation}\label{A}
0 \rightarrow \dots \rightarrow \bigoplus_{j=1}^{a_i} \mathcal{O}_X(-a_{ij}) \rightarrow \dots \rightarrow \bigoplus_{j=1}^{a_1} \mathcal{O}_X(-a_{1j}) \rightarrow \bigoplus_{j=1}^{a_0} \mathcal{O}_X(-a_{0j}) \rightarrow \mathfrak{F} \rightarrow 0.
\end{equation}
Hence
$\reg(\mathfrak{F}) = \max_{i,j} \{ a_{ij} - i \}.$
\end{note}

\begin{prop}\label{D1}
Under the above assumptions and the invertible resolution \eqref{A}, assume $a_{i,0} < a_{i,1} < \dots$ for all $i$. Then
\[
\reg(\overline{\lambda}\mathfrak{F}) = \max \{ -a_{0,1}, \reg(\mathfrak{F}^*) - 1 \}, \quad
\reg(\overline{\Tr}\mathfrak{F}) = \max \{ -a_{1,1}, \reg(\overline{\lambda}\mathfrak{F}) - 1 \}.
\]
In particular, in the case where $\mathfrak{F} \sim \mathfrak{G}$, $\reg(\mathfrak{G}) = \max \{ -a_{0,1}, \reg(\mathfrak{F}^*) - 1 \}.$
\end{prop}
\begin{proof}
Let $\mathfrak{F}^*$ have the invertible resolution
\begin{equation}\label{A1}
0 \rightarrow \dots \rightarrow \bigoplus_{j=1}^{b_1} \mathcal{O}_X(-b_{1j}) \rightarrow \bigoplus_{j=1}^{b_0}\mathcal{O}_X(-b_{0j}) \rightarrow \mathfrak{F}^* \rightarrow 0,
\end{equation}
with $b_{i,0} < b_{i,1} < \dots$ for all $i$.  
By (\ref{e2}), the invertible resolutions of $\overline{\Tr}\mathfrak{F}$ and $\overline{\lambda}\mathfrak{F}$ are
\[
\dots \rightarrow \bigoplus_{j=1}^{b_1}\mathcal{O}_X(-b_{1j}) \rightarrow \bigoplus_{j=1}^{b_0}\mathcal{O}_X(-b_{0j}) \rightarrow \bigoplus_{j=1}^{a_0}\mathcal{O}_X(a_{0j}) \rightarrow \bigoplus_{j=1}^{a_1} \mathcal{O}_X(a_{1j}) \rightarrow \overline{\Tr}\mathfrak{F} \rightarrow 0,
\]
and
\[
\dots \rightarrow \bigoplus_{j=1}^{b_1}\mathcal{O}_X(-b_{1j}) \rightarrow \bigoplus_{j=1}^{b_0}\mathcal{O}_X(-b_{0j}) \rightarrow \bigoplus_{j=1}^{a_0} \mathcal{O}_X(a_{0j}) \rightarrow \overline{\lambda}\mathfrak{F} \rightarrow 0.
\]
Thus,
$\reg(\overline{\lambda}\mathfrak{F}) = \max_j \{-a_{0,1}, b_{0,j}-1, b_{1,j}-2, \dots \} = \max \{-a_{0,1}, \reg(\mathfrak{F}^*) - 1\},$
and
\[
\reg(\overline{\Tr}\mathfrak{F}) = \max_j \{-a_{1,1}, -a_{0,1}-1, b_{0,j}-2, b_{1,j}-3, \dots \} = \max \{-a_{1,1}, \reg(\overline{\lambda}\mathfrak{F}) - 1\}.
\]
\end{proof}


\section*{Acknowledgement}
The authors would like to thank the referee for her/his comments.

This work is based upon research funded by Iran National Science Foundation (INSF) under project no.4015119. 
\bibliographystyle{amsplain}

\end{document}